\documentclass[smallextended]{svjour3_arxiv}%
\usepackage{amsmath}
\usepackage{amssymb}
\usepackage{multirow}
\usepackage{url}
\usepackage{amsfonts}
\usepackage{graphicx}
\usepackage[pdftex]{thumbpdf}
\usepackage[pdftex,
pdfstartview=FitH,bookmarks=true,bookmarksnumbered=true,bookmarksopen=true,hypertexnames=false,breaklinks=true,
colorlinks=true,  linkcolor=blue,anchorcolor=blue,
citecolor=blue,filecolor=blue,  menucolor=blue]{hyperref}%
\providecommand{\U}[1]{\protect\rule{.1in}{.1in}}
\providecommand{\U}[1]{\protect\rule{.1in}{.1in}}
\newtheorem{algorithm}[theorem]{Algorithm}
\newtheorem{assumption}[theorem]{Assumption}
\begin{document}
\title{Nested {BDDC}
for a saddle-point problem \thanks{Supported in part
by the National Science Foundation under grant \mbox{DMS-0713876},
and by the Grant Agency of the Czech Republic \mbox{GA \v{C}R 106/08/0403}.
Support from DOE/ASCR is also gratefully acknowledged.
}
}
\author{Bed\v{r}ich~Soused\'{\i}k}
\institute{B. Soused\'{\i}k \at
Department of Aerospace and Mechanical Engineering,
University of Southern California,
Los Angeles, CA 90089-2531, USA.
\email{sousedik@usc.edu}
\at Institute of Thermomechanics,
Academy of Sciences of the Czech Republic,
Dolej\v{s}kova 1402/5, 182~00 Prague~8, Czech Republic.
\at\emph{Part of the work has been completed while the author was
a Research Assistant Professor
at the Department of Mathematical and Statistical Sciences,
University of Colorado Denver.}
}
\date{}
\maketitle\begin{abstract}
We propose a Nested {BDDC} for a class of saddle-point problems. The
method solves for both flux and pressure variables. The fluxes are resolved in
three-steps: the coarse solve is followed by subdomain solves, and last we
look for a divergence-free flux correction and pressure variables using
conjugate gradients with a Multilevel BDDC preconditioner.
Because the coarse solve in the first step has the same structure as the
original problem, we can use this procedure recursively and solve
(a hierarchy of) coarse problems only approximately, utilizing the coarse problems
known from the BDDC. The resulting algorithm thus first performs several
upscaling steps, and then solves
a hierarchy of problems that have the same structure but increase in size
while sweeping down the levels, using the same components in the first
and in the
third step on each level,
and also reusing the components from the
higher levels. Because the coarsening can be quite aggressive, the number of
levels can be kept small and the additional computational cost is
significantly reduced due to the reuse of the components.
We also provide the condition number bound and numerical experiments
confirming the theory.
\end{abstract}%

\section{Introduction}

\label{sec:introduction}

The Balancing Domain Decomposition by Constraints (BDDC), proposed
independently by Cros~\cite{Cros-2003-PSC}, Dohrmann~\cite{Dohrmann-2003-PSC},
and Fragakis and Papadrakakis~\cite{Fragakis-2003-MHP}, is along with the
Finite Element Tearing and Interconnecting - Dual, Primal ({FETI-DP}) method
by Farhat et al.~\cite{Farhat-2001-FDP,Farhat-2000-SDP} currently one of the
most advanced and popular methods of iterative substructuring.
These methods have been derived by modifications of the {BDD} method by
Mandel~\cite{Mandel-1993-BDD}, and of the {FETI} method by Farhat and
Roux~\cite{Farhat-1991-MFE}, respectively. The relations between these two
families of methods have been studied extensively by many analysts in the
substructuring field cf., e.g.,
\cite{Brenner-2007-BFW,Li-2006-FBB,Mandel-2005-ATP,Mandel-2007-BFM}, and
also~\cite{Sousedik-2008-CDD}.
The methods have been also extended to multiple levels: one can find
multilevel extensions of the BDDC
in~\cite{Mandel-2008-MMB,Sousedik-2010-AMB-thesis,Sousedik-2011-AMB,Tu-2007-TBT3D,Tu-2007-TBT}
and of FETI in~\cite{Klawonn-2009-HA3}.
Here, we will be interested in extensions to saddle-point problems, such as to
the Stokes problem~\cite{Li-2006-BAI,Pavarino-2002-BNI,Sistek-2011-APB} and in
particular to the flow in porous media. One of the first domain decomposition
methods for mixed finite element problems were proposed by Glowinski and
Wheeler~\cite{Glowinski-1988-DDM}. Their Method II
has been preconditioned using BDD by Cowsar et al.~\cite{Cowsar-1995-BDD} and
using BDDC by Tu~\cite{Tu-2007-BAF}. This approach is sometimes regarded as
\emph{hybrid} because the method iterates on a system of \emph{dual} variables
(as Lagrange multipliers) enforcing the continuity of flux variables across
the substructure interfaces. However, in order to simplify 
a multilevel extension, we would like to retain the original \emph{primal} variables, and
therefore we find the recent work of Tu~\cite{Tu-2005-BAM,Tu-2011-TBA} to be
more relevant for our
approach.


In this paper, we propose a Nested {BDDC} method, which is a generalization of
the Multilevel BDDC into a larger algorithmic framework suited for a class of
saddle-point problems. Our starting point is the algorithm of Ewing and
Wang~\cite{Ewing-1992-ASA}, see also Mathew~\cite{Mathew-1993-SAIa}. The basic
idea is to solve for flux variables in three-steps: first we perform a coarse
solve which is followed by independent subdomain solves with zero boundary
conditions in the second step. In the third step, we look for a flux
correction and pressures.
Due to the design of the algorithm, the flux correction is
divergence-free, and we can use conjugate gradients (CG, resp. PCG) with a
preconditioner that preserves all of the iterates in the divergence-free
subspace. To this end we adapt the Multilevel BDDC
preconditioner\ from~\cite{Mandel-2008-MMB} to saddle-point problems.
Applications of the two-, resp. three-level BDDC in the third step of this
algorithm have been studied by Tu in~\cite{Tu-2005-BAM,Tu-2011-TBA}. Also, one
has to make a careful decision in the design of the coarse solve for the first
step. A straightforward idea is to use the same, but \textquotedblleft
coarse\textquotedblright\ finite element discretization and a natural (linear)
interpolation between the two meshes as considered
in~\cite{Mathew-1993-SAIa,Tu-2005-BAM}. Alternatively, the coarse solve has
been obtained by an action of the BDDC preconditioner on a carefully chosen
vector by Tu in~\cite{Tu-2006-BDD,Tu-2011-TBA} and she has also numerically
observed a very similar performance of the two choices~\cite[Section
4.8]{Tu-2006-BDD}. Obviously, we favor here the second idea.
Next, noting that the coarse solve in the first step has the same structure as
the original problem, we can use the algorithm recursively, and solve a
hierarchy of coarse solves only approximately.
The resulting algorithm of the Nested BDDC thus first creates a hierarchy of
(coarse) problems with similar structure scaling-up through the levels. Then
this hierarchy
is solved, while sweeping down the levels in a loop of outer iterations, using
the same components in the first and the third step on each level, and also
reusing the components from all of the previous (higher) levels. Because the
coarsening can be quite aggressive, the number of levels can be kept small and
the additional computational cost is significantly reduced due to the reusing
of components. From this perspective our method can be viewed as a way of
numerical upscaling via the coarse basis functions known from the BDDC.
Therefore, unlike some of the previous works, we do not use the global
partially assembled matrices neither the change of variables.

It is important to note that for the solution of closely related Stokes
problem, the algorithm is reduced to step three because the solution itself
is
divergence-free.
We also remark that the present approach is limited by a special choice of
finite elements. In particular, we will work with the lowest-order
Raviart-Thomas ({RT0}) elements that have piecewise constant basis functions
for pressure variables. This is not the case when, e.g., Taylor-Hood elements
are used and the BDDC preconditioned operator is no longer invariant on the
divergence-free subspace~\cite{Sistek-2011-APB}. Finally, we note that our
framework allows for irregular mesh decompositions, heterogeneous coefficients
possibly utilizing the adaptive approach as
in~\cite{Mandel-2007-ASF,Sousedik-2010-AMB-thesis}, and also allows for a
relatively straightforward extension into 3D. However, such extensions will be
studied elsewhere.

The paper is organized as follows. In Section~\ref{sec:model} we introduce the
model problem, in Section~\ref{sec:mixed} we introduce its mixed finite
element discretization and recall the original algorithm of Ewing and Wang. In
Section~\ref{sec:two-scale} we derive the two-level version of this algorithm
using the BDDC components. In Section~\ref{sec:multiscale} we formulate the
Nested BDDC method. In Section~\ref{sec:condition} we derive the condition
number bound for the model problem, and finally in Section~\ref{sec:numerical}
we report on numerical experiments with a particular application to flow in
porous media.

Throughout the paper we find it more convenient to work with abstract finite-dimensional 
spaces and linear operators between them instead of the space $%
\mathbb{R}
^{n}$ and matrices. The results can be easily converted to\ the matrix
language by choosing a finite element basis. For a symmetric positive definite
bilinear form$~a$, we will denote the energy norm by $\left\Vert u\right\Vert
_{a}=\sqrt{a\left(  u,u\right)  }$.

\section{Model problem}

\label{sec:model}

Let $\Omega$ be a bounded polygonal domain in $%
\mathbb{R}
^{n}$, $n=2$. Let us consider the following scalar, second-order, elliptic
problem given as
\begin{equation}
-\nabla\cdot k\nabla p=f,\quad\text{in }\Omega, \label{eq:problem}%
\end{equation}
where $k$ is a symmetric, uniformly positive definite matrix with bounded
coefficients, the right-hand side $f\in L^{2}\left(  \Omega\right)  $, subject
to sufficiently smooth boundary data on $\partial\Omega=\overline{\Gamma}%
_{E}\cup\overline{\Gamma}_{N}$. Equation~(\ref{eq:problem}) describes, e.g., a
pressure field in an aquifer and therefore the variable $p$ will be called
pressure. However, in reservoir simulations we are often interested in
computing $-k\nabla p$ directly.

Introducing the so-called flux variable%
\begin{equation}
\mathbf{u}=-k\nabla p, \label{eq:u}%
\end{equation}
we may rewrite~(\ref{eq:problem}) as a first-order system, generally known as
Darcy's problem,
\begin{align*}
k^{-1}\mathbf{u}+\nabla p  &  =0,\quad\text{in }\Omega,\\
\nabla\cdot\mathbf{u}  &  =f,\quad\text{in }\Omega,\\
p  &  =g_{N},\quad\text{on }\Gamma_{N},\\
\mathbf{u}\cdot\mathbf{n}  &  =g_{E},\quad\text{on }\Gamma_{E},
\end{align*}
where $\mathbf{n}$\ is the unit outward normal of $\Omega$,
and for the boundary conditions it holds that $g_{N}\in H^{1/2}\left(
\Gamma_{N}\right)  $, \ and $g_{E}\in H_{00}^{-1/2}\left(  \Gamma_{E}\right)
$. Without loss of generality, we will consider $\Gamma_{N}=\emptyset$. This
case requires a compatibility condition%
\begin{equation}
\int_{\Omega}f\,dx+\int_{\partial\Omega}g_{E}\,ds=0, \label{eq:compatibility}%
\end{equation}
and the pressure $p$ will be determined uniquely up to an additive constant.
Let us also for simplicity assume that $g_{E}=0$, and let us define a space%
\begin{equation}
\mathbf{H}_{0}(\Omega;\operatorname{div})=\left\{  \mathbf{v:v}\in
L^{2}(\Omega);\nabla\cdot\mathbf{v}\in L^{2}(\Omega)\mathbf{\quad}%
\text{and}\mathbf{\quad v}\cdot\mathbf{n}=0\;\text{on }\partial\Omega\right\}
, \label{eq:H0div}%
\end{equation}
equipped with the norm
\[
\left\Vert v\right\Vert _{\mathbf{H}_{0}(\Omega;\operatorname{div})}%
^{2}=\left\Vert \mathbf{v}\right\Vert _{L^{2}(\Omega)}^{2}+H_{\Omega}%
^{2}\left\Vert \nabla\cdot\mathbf{v}\right\Vert _{L^{2}(\Omega)}^{2},
\]
where $H_{\Omega}$ denotes the characteristic size of $\Omega$,\ and the
space
\[
L_{0}^{2}(\Omega)=\left\{  q:q\in L^{2}(\Omega)\mathbf{\quad}\text{and}%
\mathbf{\quad}\int_{\Omega}q\,dx=0\right\}  .
\]
The weak form of the Darcy's problem, we would like to solve, is
\begin{align}
\int_{\Omega}k^{-1}\mathbf{u}\cdot\mathbf{v}\,dx-\int_{\Omega}p\nabla
\cdot\mathbf{v}\,dx  &  =0,\quad\forall\mathbf{v}\in\mathbf{H}_{0}%
(\Omega;\operatorname{div}),\label{eq:problem-weak-1}\\
-\int_{\Omega}\nabla\cdot\mathbf{u}q\,dx  &  =-\int_{\Omega}fq\,dx,\quad
\forall q\in L_{0}^{2}\left(  \Omega\right)  . \label{eq:problem-weak-2}%
\end{align}
We refer to the monographs~\cite{Brezzi-1991-MHF,Toselli-2005-DDM} for
additional details.

\section{Mixed finite elements and basic algorithm}

\label{sec:mixed}

Let $U$ be the lowest order Raviart-Thomas finite element space with a zero
normal component on $\partial\Omega$ and $Q$ be a space of piecewise constants
with a zero mean on $\Omega$. These two spaces, defined on the triangulation
$\mathcal{T}_{h}$ of $\Omega$ where $h$\ denotes the mesh size, are finite-dimensional 
subspaces of $\mathbf{H}_{0}(\Omega;\operatorname{div})$ and
$L_{0}^{2}(\Omega)$, respectively, and they satisfy a uniform inf-sup
condition, see~\cite{Brezzi-1991-MHF}.

Let us define the bilinear forms and the right-hand side by%
\begin{align}
a\left(  u,v\right)   &  =\int_{\Omega}k^{-1}\mathbf{u}\cdot\mathbf{v}%
\,dx,\label{eq:a}\\
b\left(  u,q\right)   &  =-\int_{\Omega}\nabla\cdot\mathbf{u}%
q\,dx,\label{eq:b}\\
\left\langle f,q\right\rangle  &  =-\int_{\Omega}fq\,dx. \label{eq:rhs}%
\end{align}

In the mixed variational formulation of the Darcy's problem, eq.
(\ref{eq:problem-weak-1})-(\ref{eq:problem-weak-2}), we would like to find a
pair $\left(  u,p\right)  \in\left(  U,Q\right)  $ such that
\begin{align}
a\left(  u,v\right)  +b\left(  v,p\right)   &  =0,\qquad\forall v\in
U,\label{eq:variational-1}\\
b\left(  u,q\right)   &  =\left\langle f,q\right\rangle ,\qquad\forall q\in
Q.\label{eq:variational-2}%
\end{align}
Let us split the domain $\Omega$ into non-overlapping subdomains $\Omega_{i}$,
$i=1,\dots,N$, assuming further that they form a triangulation of$~\Omega$,
e.g., for a moment as macroelements. Accordingly, let us split the solution
spaces as%
\begin{align}
U &  =U_{0}+\left(  \oplus_{i=1}^{N}U_{i}\right)  +U_{\text{corr}%
},\label{eq:U_decomposition}\\
Q &  =\oplus_{i=0}^{N}Q_{i}.\label{eq:P_decomposition}%
\end{align}
The spaces $U_{0}$, $Q_{0}$ are obtained by considering subdomains as
macroelements. The spaces $U_{i}$, $Q_{i}$, for $i=1,\dots,N$, are obtained by
a restriction from the global solution spaces $U$, $Q$. More specifically,
because $U_{I}=\oplus_{i=1}^{N}U_{i}$, the functions from $U_{i}$ have
vanishing normal components (i.e., zero fluxes) along the subdomain
interfaces. Also, in order to determine the pressure $p$ uniquely, we will
consider the component $p_{0}\in Q_{0}$, which is constant in each
subdomain$~\Omega_{i}$, to have a zero average over the whole domain $\Omega$,
and the components $p_{i}\in Q_{i}$ to have zero averages over the subdomain
$\Omega_{i}$ and identically equal to zero in other subdomains. The
introduction of the auxiliary space $U_{\text{corr}}=U$ is motivated by an
observation that in general%
\begin{equation}
U\neq U_{0}+\left(  \oplus_{i=1}^{N}U_{i}\right)
,\label{eq:U_decomposition_neq}%
\end{equation}
because the fluxes on subdomain interfaces might not be constant. We note that
we will take an advantage of this splitting, in particular because for all
$u_{I}\in U_{I}$\ and $q_{0}\in Q_{0}$, it holds, by the divergence theorem,
that%
\begin{equation}
b\left(  u_{I},q_{0}\right)  =-\int_{\Omega}\left(  \nabla\cdot u_{I}\right)
\,q_{0}\,dx=0.\label{eq:b-orth}%
\end{equation}
The following algorithm is due to Ewing and Wang~\cite{Ewing-1992-ASA}, cf.
also Mathew~\cite{Mathew-1993-SAIa}.

\begin{algorithm}
[Basic]\label{alg:basic} Find the pair $\left(  u,p\right)  \in\left(
U,Q\right)  $\ that satisfies (\ref{eq:variational-1})-(\ref{eq:variational-2}%
) as%
\[
u=u_{0}+\sum_{i=1}^{N}u_{i}+u_{\text{corr}},
\]
in the following three steps: Compute

\begin{enumerate}
\item the coarse component $\left(  u_{0},p_{0}\right)  \in\left(  U_{0}%
,Q_{0}\right)  $ by solving%
\begin{align}
a\left(  u_{0},v_{0}\right)  +b\left(  v_{0},p_{0}\right)   &  =0,\qquad
\forall v_{0}\in U_{0},\label{eq:u_0}\\
b\left(  u_{0},q_{0}\right)   &  =\left\langle f,q_{0}\right\rangle
,\qquad\forall q_{0}\in Q_{0}. \label{eq:p_0}%
\end{align}
Note that because $Q_{0}\varsubsetneq Q$, in general%
\[
b\left(  u_{0},q\right)  \neq\left\langle f,q\right\rangle ,\qquad\forall q\in
Q.
\]

\item the substructure components $\left(  u_{i},p_{i}\right)  \in\left(
U_{i},Q_{i}\right)  $ for $i=1,\dots,N$\ from%
\begin{align*}
a\left(  u_{i},v_{i}\right)  +b\left(  v_{i},p_{i}\right)   &  =-a\left(
u_{0},v_{i}\right)  ,\qquad\forall v_{i}\in U_{i},\\
b\left(  u_{i},q_{i}\right)   &  =\left\langle f,q_{i}\right\rangle -b\left(
u_{0},q_{i}\right)  ,\qquad\forall q_{i}\in Q_{i}.
\end{align*}
\qquad\qquad

Add the computed solutions as%
\[
u^{\ast}=u_{0}+\sum_{i=1}^{N}u_{i}.
\]
Due to the correction in the second step, and with respect to
(\ref{eq:P_decomposition}), we obtain%
\begin{equation}
b\left(  u^{\ast},q\right)  =\left\langle f,q\right\rangle, \qquad\forall q\in
Q. \label{eq:u^star-in_b}%
\end{equation}
On the other hand, from (\ref{eq:U_decomposition_neq}), in general $u^{\ast
}\neq u$. Therefore, we also need

\item the correction $u_{\text{corr}}\in U_{\text{corr}}=U$. Considering
\[
u=u^{\ast}+u_{\text{corr}},
\]
substituting into (\ref{eq:variational-1})-(\ref{eq:variational-2}) and using
(\ref{eq:u^star-in_b}), compute $\left(  u_{\text{corr}},p\right)  \in\left(
U,Q\right)  $ from%
\begin{align*}
a\left(  u_{\text{corr}},v\right)  +b\left(  v,p\right)   &  =-a\left(
u^{\ast},v\right)  ,\qquad\forall v\in U,\\
b\left(  u_{\text{corr}},q\right)   &  =0,\qquad\forall q\in Q.
\end{align*}

\end{enumerate}
\end{algorithm}

\begin{remark}
We would like to accentuate the reduction effect of Algorithm~\ref{alg:basic}:
the structural difference between problem (\ref{eq:variational-1})-(\ref{eq:variational-2})
and the problem in Step~3 of Algorithm~\ref{alg:basic} is that the right-hand side of the reduced
problem has a vanishing second component, which corresponds to the divergence-free subspace.
Also, because the pressure components $p_{0}$, $p_{I}$ computed in
Step 1 and Step 2, resp., are tested only against proper subspaces of $U$, we simply disregard
them.
\end{remark}

The application of the BDDC\ preconditioner for the computation of
$u_{\text{corr}}$ for the two-, resp. three-level BDDC method has been studied
by Tu~\cite{Tu-2005-BAM,Tu-2011-TBA}. However, comparing (\ref{eq:u_0}%
)-(\ref{eq:p_0}) with (\ref{eq:variational-1})-(\ref{eq:variational-2}), we
see that in fact we can use the same algorithm recursively, with multiple
levels,
to solve for both $u_{0}$\ and $u_{\text{corr}}$. But first, let us
reformulate the basic Algorithm~\ref{alg:basic} with {BDDC components}.

\section{Basic algorithm with {BDDC components}}

\label{sec:two-scale}

We begin by introducing the substructuring components.
Let $\Omega$ be decomposed into nonoverlapping subdomains $\Omega_{i}$,
$i=1,\ldots,N,$ also called substructures, forming a quasi-uniform
triangulation of $\Omega$ with the characteristic subdomain size~$H$. Each
substructure is a union of the lowest order Raviart-Thomas ({RT0}) finite
elements with a matching discretization across the substructure interfaces.
Let $\Gamma_{i}=\partial\Omega_{i}\backslash\partial\Omega$\ be the set of
boundary degrees of freedom of the substructure$~\Omega_{i}$\ shared with
other substructures$~\Omega_{j}$, $j\neq i$, and let us define the interface
by $\Gamma=\cup_{i=1}^{N}\Gamma_{i}$. Let us denote by $\mathcal{F}$ the set
of all faces between substructures, i.e., in the present context the set of
all intersections $\Gamma_{ij}=\Gamma_{i}\cap\Gamma_{j}$, $i\neq j$. Note that
with respect to our discretization we define only \emph{faces}, but no
\emph{corners} (nor \emph{edges} in 3D) known from other types of
substructuring. Let us also slightly generalize the settings by allowing for
constant coefficients $k_{i}$\ in each subdomain$~\Omega_{i}$ separately.

Let us consider, cf. eq.~(\ref{eq:P_decomposition}), the decomposition of the
pressure space
\begin{equation}
Q=Q_{0}\oplus Q_{I}\text{,\quad and\quad}Q_{I}=Q_{1}\times\dots\times Q_{N},
\label{eq:Q-dec}%
\end{equation}
where $Q_{0}$ consists of constant functions in each subdomain, such that%
\[
\int_{\Omega}q_{0}\,dx=0,\quad\forall q_{0}\in Q_{0},\quad\text{and}\quad
\int_{\Omega_{i}}q_{i}\,dx=0,\quad\forall q_{i}\in Q_{i},\;i=1,\dots,N.
\]
Again, the space $Q$ is a finite-dimensional subspace of~$L^2_0 \left (\Omega \right)$, 
and therefore the unique solvability of all subsequently considered mixed problems is guaranteed.

Next, let $W_{i}$ be the space of the flux finite element functions on a
substructure $\Omega_{i}$ such that all of their degrees of freedom on
$\partial\Omega_{i}\cap\partial\Omega$ are zero, and let
\[
W=W_{1}\times\dots\times W_{N}.
\]
Now $U\subset W$ can be viewed as the subspace of all functions from~$W$
continuous across substructure interfaces. Define $U_{I}\subset U$ as the
subspace of functions that are zero on the interface $\Gamma$,
i.e., the space of \textquotedblleft interior\textquotedblright\ functions and
let us define a projection$~P:w\in W\longmapsto\left(  u_{I},p_{I}\right)
\in\left(  U_{I},Q_{I}\right)  $ such that
\begin{align*}
a\left(  u_{I},v_{I}\right)  +b\left(  v_{I},p_{I}\right)   &  =a\left(
w,v_{I}\right)  ,\quad\forall v_{I}\in U_{I},\\
b\left(  u_{I},q_{I}\right)   &  =b\left(  w,q_{I}\right)  ,\quad\forall
q_{I}\in Q_{I}.
\end{align*}
Let us also define a projection$~P_{a}:w\in W\longmapsto u_{I}\in U_{I}%
\ $\ such that%
\[
a\left(  u_{I},v_{I}\right)  =a\left(  w,v_{I}\right)  ,\quad\forall v_{I}\in
U_{I}.
\]
Functions from the nullspace of $P$ and $P_{a}$\ will be called Stokes
harmonic and discrete harmonic, respectively. The following comparison of
their energies, cf.~\cite[Lemma 9.10]{Toselli-2005-DDM}, will allow us to
apply some arguments from the scallar elliptic theory
in~\cite{Mandel-2008-MMB} to the saddle-point problem considered here.

\begin{lemma}
\label{lem:norm-equiv}
Let $w\in W$. Then,%
\[
c\left\Vert \left(  I-P\right)  w\right\Vert _{a}\leq\left\Vert \left(
I-P_{a}\right)  w\right\Vert _{a}\leq\left\Vert \left(  I-P\right)
w\right\Vert _{a}.%
\]
\end{lemma}

Next, let $\widehat{W}$ be the space of all Stokes harmonic functions that are
continuous across substructure interfaces, and such that
\begin{equation}
U=U_{I}\oplus\widehat{W},\quad\text{and\quad}U_{I}\perp_{a}\widehat{W}.
\label{eq:int-harm-dec}%
\end{equation}

The first step in substructuring is typically the reduction of the problem to
the interfaces. In particular, let us consider Step~3 of
Algorithm~\ref{alg:basic}, which can be written a bit more generally as: find
a pair $\left(  u,p\right)  \in\left(  U,Q\right)  $ such that
\begin{align}
a\left(  u,v\right)  +b\left(  v,p\right)   &  =\left\langle f^{\ast
},v\right\rangle ,\quad\forall v\in U,\label{eq:Stokes-full-1}\\
b\left(  u,q\right)   &  =0,\quad\forall q\in Q. \label{eq:Stokes-full-2}%
\end{align}
The problem (\ref{eq:Stokes-full-1})-(\ref{eq:Stokes-full-2}) can be reduced
to finding $\left(  \widehat{w},p_{0}\right)  \in\left(  \widehat{W}%
,Q_{0}\right)  $ such that
\begin{align}
a\left(  \widehat{u},\widehat{v}\right)  +b\left(  \widehat{v},p_{0}\right)
&  =\left\langle f^{\ast},\widehat{v}\right\rangle ,\quad\forall\widehat{v}%
\in\widehat{W},\label{eq:Stokes-reduced-1}\\
b\left(  \widehat{u},q_{0}\right)   &  =0,\quad\forall q_{0}\in Q_{0}.
\label{eq:Stokes-reduced-2}%
\end{align}
Such \textquotedblleft reduction\textquotedblright\ is in implementation
achieved by elimination of the interiors, known also as static condensation,
see, e.g.,~\cite[Section 9.4.2]{Toselli-2005-DDM} for more details. Now, let
us define a subspace of \emph{balanced} functions
as%
\begin{equation}
\widehat{W}_{B}=\left\{  \widehat{v}\in\widehat{W}:b\left(  \widehat{v}%
,q_{0}\right)  =0,\quad\forall q_{0}\in Q_{0}\right\}  .
\label{eq:W_hat_balanced}%
\end{equation}
The problem (\ref{eq:Stokes-reduced-1})-(\ref{eq:Stokes-reduced-2}) is
equivalent to the following positive definite problem
\begin{equation}
\widehat{u}\in\widehat{W}_{B}:\quad a\left(  \widehat{u},\widehat{v}\right)
=\left\langle f^{\ast},\widehat{v}\right\rangle ,\quad\forall\widehat{v}%
\in\widehat{W}_{B}. \label{eq:Stokes-balanced}%
\end{equation}
Note that the space $U_{I}$ is balanced due to~(\ref{eq:b-orth}). Then, using
$\widehat{W}_{B}$ in the splitting~(\ref{eq:int-harm-dec}) implies that $U$ is
also balanced in the sense of the definition~(\ref{eq:W_hat_balanced}).

The BDDC\ method is a two-level preconditioner characterized by the selection
of certain \emph{coarse degrees of freedom}. In the present setting these will
be flux averages over each face, and pressure averages over each substructure,
cf. Assumption~\ref{ass:enough-constraints}.\ In particular, the value of a
coarse degree of freedom will be taken as an average of the fine scale degrees
of freedom. Next, let $\widetilde{W}\subset W$ be the subspace of all
functions such that the values of any flux coarse degrees of freedom have a
common value over a face shared by a pair of adjacent substructures, and
vanish on $\partial\Omega$. Next, define $\widetilde{W}_{\Pi}\subset
\widetilde{W}$ as the subspace of all functions such that their flux coarse
degrees of freedom between pairs of adjacent substructures coincide, and such
that they are Stokes harmonic, and let us also define $\widetilde{W}_{\Delta
}\subset W$ as the subspace of all functions such that their flux coarse
degrees of freedom vanish. Clearly, functions in $\widetilde{W}_{\Pi}$ are
uniquely determined by the values of their flux coarse degrees of freedom,
and
\begin{equation}
\widetilde{W}=\widetilde{W}_{\Delta}\oplus\widetilde{W}_{\Pi}.
\label{eq:tilde-dec}%
\end{equation}
Let $E$ be a projection from $\widetilde{W}$ onto $U$, defined by
taking some weighted average of corresponding degrees of freedom on
substructure interfaces, cf. Remark~\ref{rem:averaging}.

\begin{remark}
\label{rem:averaging} The entries in the matrix corresponding to the averaging
operator$~E$ are given by scaling weights corresponding to a degree of freedom $x \in \Omega_i$ as
\[
e_{i}(x)=%
\begin{cases}
\frac{k_{i}^{-\gamma}}{k_{i}^{-\gamma}+k_{j}^{-\gamma}} & \mbox{if }x\in
\partial\Omega_{i}\cap\partial\Omega_{j},\\
1 & \mbox{if }x\in\Omega_{i}\backslash\Gamma.
\end{cases}
\]
The case $\gamma=1$ corresponds to the so-called $\rho-$scaling,
$\gamma=0$, i.e. $e_{i}\left(  x\right)  $ is $1/2$ or $1,$ corresponds to the
multiplicity scaling, cf.~\cite{Klawonn-2008-AFA}. We note that the $\rho
-$scaling is the same as the stiffness scaling because each flux degree of freedom
is shared by two elements shared by at most a pair of subdomains.
\end{remark}

Next, observe that it is only required for $u^{\ast}$ to satisfy
(\ref{eq:u^star-in_b}). In particular, we do not need the substructures to
form the same discretization as on the finite element level. Instead, we can
conveniently retain the algebraic framework of the BDDC\ method introduced
above and use its coarse problem in place of the coarse solve in Step~1.
Specifically, let us set $U_{0}=\widetilde{W}_{\Pi}$. We are now ready to take
the second look at Algorithm~\ref{alg:basic} and formulate its first modification.

\begin{algorithm}
[Basic algorithm with BDDC\ components]\label{alg:two-scale}Find the solution
$\left(  u,p\right)  \in\left(  U,Q\right)  $\ of the problem
(\ref{eq:variational-1})-(\ref{eq:variational-2}) by computing:

\begin{enumerate}
\item the coarse component $u_{0}\in\widehat{W}:$
solving for $\left(  \widetilde{w}_{0},p_{0}\right)  \in\left(  \widetilde{W}%
_{\Pi},Q_{0}\right)  $\ the system
\begin{align}
a\left(  \widetilde{w}_{0},\widetilde{v}_{\Pi}\right)  +b\left(
\widetilde{v}_{\Pi},p_{0}\right)   &  =0,\qquad\forall\widetilde{v}_{\Pi}%
\in\widetilde{W}_{\Pi},\label{eq:two-scale-coarse-1}\\
b\left(  \widetilde{w}_{0},q_{0}\right)   &  =\left\langle f,q_{0}%
\right\rangle ,\qquad\forall q_{0}\in Q_{0}, \label{eq:two-scale-coarse-2}%
\end{align}
dropping $p_{0}$, and applying the projection%
\[
u_{0}=E\widetilde{w}_{0}.
\]

\item the substructure components $\left(  u_{I},p_{I}\right)  \in\left(
U_{I},Q_{I}\right)  $ solving%
\begin{align*}
a\left(  u_{I},v_{I}\right)  +b\left(  v_{I},p_{I}\right)   &  =-a\left(
u_{0},v_{I}\right)  ,\qquad\forall v_{I}\in U_{I},\\
b\left(  u_{I},q_{I}\right)   &  =\left\langle f,q_{I}\right\rangle -b\left(
u_{0},q_{I}\right)  ,\qquad\forall q_{I}\in Q_{I},
\end{align*}
\qquad\qquad

dropping $p_{I}$, and combining the solutions $u^{\ast}=u_{0}+u_{I}.$

\item the correction and the pressure $\left(  u_{\text{corr}},p\right)
\in\left(  U,Q\right)  $ from%
\begin{align*}
a\left(  u_{\text{corr}},v\right)  +b\left(  v,p\right)   &  =-a\left(
u^{\ast},v\right)  ,\qquad\forall v\in U,\\
b\left(  u_{\text{corr}},q\right)   &  =0,\qquad\forall q\in Q.
\end{align*}
Specifically, use the PCG method with the two-level BDDC preconditioner
defined in Algorithm~\ref{alg:two-level}, using the coarse problem
(\ref{eq:two-scale-coarse-1})-(\ref{eq:two-scale-coarse-2}).
\end{enumerate}

Finally, combine the three solutions as
\[
u=u_{0}+u_{I}+u_{\text{corr}}.
\]

\end{algorithm}

Note that we again disregard the pressures $p_{0}$ and $p_{I}$ from Steps~1
and~2 as in Algorithm~\ref{alg:basic}. The algorithm of the two-level
BDDC\ preconditioner used in Step~3
is closely related to the original version for elliptic problems,
cf.~\cite[Algorithm~11]{Mandel-2008-MMB}. For completeness its version for
saddle-point problems follows.

\begin{algorithm}
[Two-level BDDC preconditioner]\label{alg:two-level} Define the preconditioner
$\left(  r,0\right)  \in\left(  U^{\prime},Q^{\prime}\right)  \longmapsto
\left(  u,p\right)  \in\left(  U,Q\right)  $\ as follows:

\noindent Compute the interior pre-correction $\left(  u_{I},p_{I}\right)
\in\left(  U_{I},Q_{I}\right)  $ from%
\begin{align*}
a\left(  u_{I},z_{I}\right)  +b\left(  z_{I},p_{I}\right)   &  =\left\langle
r,z_{I}\right\rangle ,\qquad\forall z_{I}\in U_{I},\\
b\left(  u_{I},q_{I}\right)   &  =0,\qquad\forall q_{I}\in Q_{I}.
\end{align*}
Set up the updated residual%
\[
r_{B}\in U^{\prime},\quad\left\langle r_{B},v\right\rangle =\left\langle
r,v\right\rangle -\left[  a\left(  u_{I},v\right)  +b\left(  v,p_{I}\right)
\right]  ,\qquad\forall v\in U.
\]
Compute the substructure correction $w_{\Delta}\in\widetilde{W}_{\Delta}$
from
\begin{align*}
a\left(  w_{\Delta},z_{\Delta}\right)  +b\left(  z_{\Delta},p_{I\Delta
}\right)   &  =\left\langle r_{B},Ez_{\Delta}\right\rangle ,\qquad\forall
z_{\Delta}\in\widetilde{W}_{\Delta},\\
b\left(  w_{\Delta},q_{I}\right)   &  =0,\qquad\forall q_{I}\in Q_{I}.
\end{align*}
Compute the coarse correction $\left(  w_{\Pi},p_{0}\right)  \in\left(
\widetilde{W}_{\Pi},Q_{0}\right)  $ from
\begin{align*}
a\left(  w_{\Pi},z_{\Pi}\right)  +b\left(  z_{\Pi},p_{0}\right)   &
=\left\langle r_{B},Ez_{\Pi}\right\rangle ,\qquad\forall z_{\Pi}%
\in\widetilde{W}_{\Pi},\\
b\left(  w_{\Pi},q_{0}\right)   &  =0,\qquad\forall q_{0}\in Q_{0}.
\end{align*}
Add the averaged corrections%
\[
u_{B}=E\left(  w_{\Delta}+w_{\Pi}\right)  .
\]
Compute the interior post-correction $\left(  v_{I},q_{I}\right)  \in\left(
U_{I},Q_{I}\right)  $ from%
\begin{align*}
a\left(  v_{I},z_{I}\right)  +b\left(  z_{I},q_{I}\right)   &  =a\left(
u_{B},z_{I}\right)  ,\quad\forall z_{I}\in U_{I},\\
b\left(  v_{I},\overline{q}_{I}\right)   &  =b\left(  u_{B},\overline{q}%
_{I}\right)  ,\quad\forall\overline{q}_{I}\in Q_{I}.
\end{align*}
Apply the combined corrections%
\begin{align*}
u  &  =u_{I}+u_{B}-v_{I},\\
p  &  =p_{I}+p_{0}-q_{I}.
\end{align*}

\end{algorithm}

\begin{remark}
The solve in the space $\widetilde{W}_{\Delta}$\ gives rise to independent
problems on substructures and the global coarse problem in the space
$\widetilde{W}_{\Pi}$ is exactly the same as the one used in Step 1
of Algorithm~\ref{alg:two-scale}.
\end{remark}

We could implement Step 3 of Algorithm~\ref{alg:two-scale} by performing first
the static condensation, iteratively solving the problem in the spaces
$\left(  \widehat{W},Q_{0}\right)  $,\ and recovering the interiors after the
convergence. This would remove the interior pre-, and post-corrections from
Algorithm~\ref{alg:two-level}, cf.~\cite[Algorithms~7, 9, 11]{Mandel-2008-MMB}%
, but performance of these two versions would be the same,
cf.~\cite[Theorem~14]{Mandel-2008-MMB}. Such approach might be also more
appealing from the practical point of view, because it allows for iterations
on a much smaller, Schur complement, system of linear equations see,
e.g.,~\cite[Sections 4.3 and 9.4.2]{Toselli-2005-DDM} for details. For a proof
that given a sufficient number of constraints, the PCG\ method with the
two-level BDDC preconditioner is invariant on the space of balanced, resp.
divergence-free functions see~\cite[Lemma~2]{Tu-2005-BAM} or
Lemma~\ref{lem:divergence-free}\ in the next section.

In order to provide the condition number bound, let us introduce a larger
space of \emph{balanced} functions defined as%
\[
\widetilde{W}_{B}=\left\{  v\in\widetilde{W}:b\left(  v,q_{0}\right)
=0,\quad\forall q_{0}\in Q_{0}\right\}  ,
\]
i.e., $\widehat{W}_{B}\subset\widetilde{W}_{B}$, and for which we get, using
(\ref{eq:H0div}) and (\ref{eq:a}), the equivalence%
\begin{equation}
c\left\Vert v\right\Vert _{a}^{2}\leq\left\Vert v\right\Vert _{\mathbf{H}%
_{0}(\Omega;\operatorname{div})}^{2}\leq C\left\Vert v\right\Vert _{a}%
^{2},\quad\forall v\in\widetilde{W}_{B}. \label{eq:norm-equiv}%
\end{equation}

Due to the equivalence of the problems (\ref{eq:Stokes-full-1}%
)-(\ref{eq:Stokes-full-2}), (\ref{eq:Stokes-reduced-1}%
)-(\ref{eq:Stokes-reduced-2}) and (\ref{eq:Stokes-balanced}), and with respect
to the equivalence of norms (\ref{eq:norm-equiv}) and
Lemma~\ref{lem:norm-equiv}, we can conveniently use the $a-$norm in the
following estimate, and the condition number bound known from the elliptic
case cf., e.g.,~\cite[Theorem 4]{Mandel-2007-BFM} carries over.

\begin{theorem}
[{\cite[Lemma~8, Theorem~1]{Tu-2005-BAM}}]\label{thm:two-level-bound}The condition
number~$\kappa$ of the two-level BDDC preconditioner from
Algorithm~\ref{alg:two-level} satisfies the bound
\begin{equation}
\kappa\leq\omega=\max\left\{  {\sup_{w\in\widetilde{W}_{B}}\frac{\left\Vert
\left(  I-P\right)  Ew\right\Vert _{a}^{2}}{\left\Vert w\right\Vert _{a}^{2}%
},1}\right\}  \text{ }\leq C\left(  1+\log\frac{H}{h}\right)  ^{2}%
{.}\label{eq:two-scale-bound}%
\end{equation}
\end{theorem}

\begin{remark}
\label{rem:two-scale-bound}In~\cite[Lemma 8]{Tu-2005-BAM}, the supremum was
taken over the space $\left(  I-P\right)  \widetilde{W}_{B}$ of Stokes
harmonic balanced function. Nevertheless, the bound remains the same by
considering the larger space $\widetilde{W}_{B}$, cf. also~\cite[Remark
16]{Mandel-2008-MMB}. 
\end{remark}

In Algorithm~\ref{alg:two-scale}, the coarse problem used in Steps 1 and 3 is
solved exactly, and therefore becomes a bottleneck in the case of many
substructures. In the next section we will suggest its further modification by
using it recursively for Step 1, on a multiple of different levels leading to
the Nested BDDC method.

\section{Nested {BDDC}}

\label{sec:multiscale}We extend Algorithm~\ref{alg:two-scale} to multiple
levels by using it recursively for Step 1, leading to a multilevel
decomposition, and introducing thus a loop of outer iterations with the size
given by the number of different decomposition levels.

\begin{figure}[ptbh]
\begin{center}
\includegraphics[width=7.4cm]{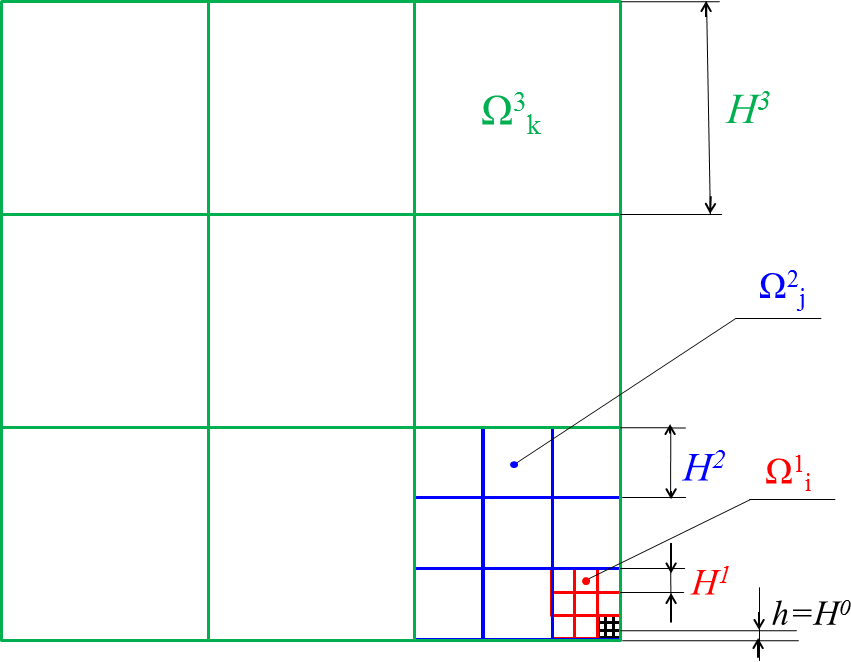}
\end{center}
\caption{An example of a uniform decomposition for a four-level method with
$H^{\ell}/H^{\ell-1}=3$.}%
\label{fig1}%
\end{figure}

\begin{figure}[ptb]
\medskip%
\[
\fbox{$%
\begin{array}
[c]{ccccccccccccccc}
&  & Q & = & Q^{0}_{0} &  &  &  &  &  &  &  &  &  & \\
&  & \shortparallel &  &  &  &  &  &  &  &  &  &  &  & \\
&  & Q^{1} & = & Q^{1}_{0} & \oplus & Q^{1}_{I} &  &  &  &  &  &  &  & \\
&  &  &  & \shortparallel &  &  &  &  &  &  &  &  &  & \\
&  &  &  & Q^{2} & = & Q^{2}_{0} & \oplus & Q^{2}_{I} &  &  &  &  &  & \\
&  &  &  &  &  & \shortparallel &  &  &  &  &  &  &  & \\
&  &  &  &  &  & \vdots &  &  &  &  &  &  &  & \\
&  &  &  &  &  & \shortparallel &  &  &  &  &  &  &  & \\
&  &  &  &  &  & Q^{L-1} & = & Q^{L-1}_{0} & \oplus & Q^{L-1}_{I} &  &  &  &
\\
&  &  &  &  &  &  &  &  &  &  &  &  &  & \\
&  & U & = & \widetilde{W}_{\Pi}^{0} &  &  &  &  &  &  &  &  &  & \\
&  & \shortparallel &  &  &  &  &  &  &  &  &  &  &  & \\
U_{I}^{1} & {\genfrac{}{}{0pt}{}{\genfrac{}{}{0pt}{}{P^{1}}{\leftarrow
}}{\subset}} & U^{1} & {\genfrac{}{}{0pt}{}{\genfrac{}{}{0pt}{}{E^{1}%
}{\leftarrow}}{\subset}} & \widetilde{W}_{\Pi}^{1} & \oplus & \widetilde{W}%
_{\Delta}^{1} & = & \widetilde{W}^{1} & \subset & W^{1} &  &  &  & \\
&  &  &  & \shortparallel &  &  &  &  &  &  &  &  &  & \\
&  & U_{I}^{2} & {\genfrac{}{}{0pt}{}{\genfrac{}{}{0pt}{}{P^{2}}{\leftarrow
}}{\subset}} & U^{2} & {\genfrac{}{}{0pt}{}{\genfrac{}{}{0pt}{}{E^{2}%
}{\leftarrow}}{\subset}} & \widetilde{W}_{\Pi}^{2} & \oplus & \widetilde{W}%
_{\Delta}^{2} & = & \widetilde{W}^{2} & \subset & W^{2} &  & \\
&  &  &  & {\scriptstyle\ \downarrow I^{2}} &  & \shortparallel &  &  &  &  &
&  &  & \\
&  &  &  & \widetilde{U}^{2} &  & \vdots &  &  &  &  &  &  &  & \\
&  &  &  &  &  & \shortparallel &  &  &  &  &  &  &  & \\
&  &  &  & U_{I}^{L-1} & {\genfrac{}{}{0pt}{}{\genfrac{}{}{0pt}{}{P^{L-1}%
}{\leftarrow}}{\subset}} & U^{L-1} &
{\genfrac{}{}{0pt}{}{\genfrac{}{}{0pt}{}{E^{L-1}}{\leftarrow}}{\subset}} &
\widetilde{W}_{\Pi}^{L-1} & \oplus & \widetilde{W}_{\Delta}^{L-1} & = &
\widetilde{W}^{L-1} & \subset & W^{L-1}\\
&  &  &  &  &  & {\scriptstyle\ \downarrow I^{L-1}} &  & \shortparallel &  &
&  &  &  & \\
&  &  &  &  &  & \widetilde{U}^{L-1} &  & U^{L} &  &  &  &  &  & \\
&  &  &  &  &  &  &  & {\scriptstyle\ \downarrow I^{L}} &  &  &  &  &  & \\
&  &  &  &  &  &  &  & \widetilde{U}^{L} &  &  &  &  &  &
\end{array}
$}%
\]
\medskip\caption{Space decompositions, embeddings and projections in the
Nested and Multilevel BDDC for a saddle-point problem described in
Algorithm~\ref{alg:multiscale} and Algorithm~\ref{alg:multilevel-bddc},
respectively. Note that the spaces~$\protect\widetilde{W}_{\Pi}^{\ell}$,
$\ell=1,\dots,L-1$ in the Multilevel BDDC are by (\ref{eq:b_wPi_q0}) also
balanced. However in order to guarantee that the output of the Multilevel BDDC
preconditioner is also balanced, resp. divergence-free in the sense of
eq.~(\ref{eq:Stokes-full-2}), we need to satisfy
Assumption~\ref{ass:enough-constraints}.}%
\label{fig:multi-bddc}%
\end{figure}

The substructuring components from Section~\ref{sec:two-scale} will be denoted
by an additional superscript~$^{1},$ as~$\Omega_{i}^{1},$ $i=1,\ldots N^{1}$,
etc., and called level~$1$. In particular, the problem (\ref{eq:variational-1}%
)-(\ref{eq:variational-2}) will be denoted as: find $\left(  u^{1}%
,p^{1}\right)  \in\left(  U^{1},Q^{1}\right)  $\ such that
\begin{align}
a\left(  u^{1},v^{1}\right)  +b\left(  v^{1},p^{1}\right)   &  =0,\qquad
\forall v^{1}\in U^{1},\label{eq:variational-level1-1}\\
b\left(  u^{1},q^{1}\right)   &  =\left\langle f^{1},q^{1}\right\rangle
,\qquad\forall q^{1}\in Q^{1}, \label{eq:variational-level1-2}%
\end{align}
The level~$1$ coarse problem solved in (\ref{eq:two-scale-coarse-1}%
)-(\ref{eq:two-scale-coarse-2}) will be called the level~$2$ problem. It has
the same finite element structure as the original problem
(\ref{eq:variational-1})-(\ref{eq:variational-2}) on level~$1$, so we put
$\widetilde{W}_{\Pi}^{1}=U^{2}$ and $Q_{0}^{1}=Q^{2}$. Level~$1$ substructures
are level~$2$ elements and level $1$ coarse degrees of freedom are level~$2$
degrees of freedom. Repeating this process recursively, level~$\ell-1$
substructures become level~$\ell$ elements, and the level~$\ell$ substructures
are agglomerates of level~$\ell$ elements. An $L-$level method is thus given
by nested decomposition levels $\ell=1,\dots,L-1$. Level~$\ell$ substructures
are denoted by $\Omega_{i}^{\ell},$ $i=1,\ldots,N^{\ell},$ and they are
assumed to form a conforming triangulation with a characteristic substructure
size $H^{\ell}$. An example of a decomposition is in Figure~\ref{fig1}. For
convenience, we denote by~$\Omega_{i}^{0}$ the original finite elements and
put $H^{0}=h$. The interface$~\Gamma^{\ell}$\ on level$~\ell$ is defined as
the union of all level$~\ell$ boundary degrees of freedom, i.e., degrees of
freedom shared by at least two level$~\ell$ substructures, and we note that
$\Gamma^{\ell}\subset\Gamma^{\ell-1}$. Level$~\ell-1$ coarse degrees of
freedom become level$~\ell$ degrees of freedom. The shape functions on
level~$\ell$ are Stokes harmonic with respect to level~$\ell-1$ shape
functions, subject to the value of exactly one level$~\ell$ degree of freedom
being one and others level$~\ell$ degrees of freedom being zero. We remark
that as before the coarse degrees of freedom will be the flux averages over
each face, and pressure averages over each substructure, cf.
Assumption~\ref{ass:enough-constraints}.\ The (Stokes harmonic) projection is
performed on each level$~\ell$ element (level$~\ell-1$ substructure)
separately, so the values of level$~\ell-1$ degrees of freedom are in general
discontinuous between level$~\ell-1$ substructures, and only the values of
level$~\ell$\ degrees of freedom between neighboring level~$\ell$\ elements coincide.

The development of the spaces on level $\ell$ now parallels the finite element
setting in Section \ref{sec:two-scale}, see also~\cite[Section 6]%
{Mandel-2008-MMB}. First, let us consider similarly as before,
cf.~eq.~(\ref{eq:Q-dec}),\ the recursive decomposition of the pressure spaces
\begin{equation}
Q^{\ell}=Q_{0}^{\ell}\oplus Q_{I}^{\ell}\text{,\quad and\quad}Q_{I}^{\ell
}=Q_{1}^{\ell}\times\cdots\times Q_{N^{\ell}}^{\ell},\text{\qquad}\ell
=1,\dots,L-1, \label{eq:Q-dec-levels}%
\end{equation}
where $Q_{0}^{\ell}$ consists of constant functions in each level $\ell$
substructure, such that%
\[
\int_{\Omega^{\ell}}q_{0}^{\ell}\,dx=0,\quad\forall q_{0}^{\ell}\in Q_{0}^{\ell},
\quad\text{and}\quad\int_{\Omega_{i}^{\ell}}q_{i}^{\ell}\,dx=0,\quad\forall
q_{i}^{\ell}\in Q_{i}^{\ell},\;i=1,\dots,N^{\ell}.
\]

Next, denote $U^{\ell}=\widetilde{W}_{\Pi}^{\ell-1}$. Let $W_{i}^{\ell}$ be
the space of the flux functions\ on the substructure $\Omega_{i}^{\ell}$, such
that all of their degrees of freedom on $\partial\Omega_{i}^{\ell}\cap
\partial\Omega$ are zero, and on each decomposition level$~\ell=1,\dots,L-1,$
let%
\[
W^{\ell}=W_{1}^{\ell}\times\cdots\times W_{N^{\ell}}^{\ell}.
\]
Now $U^{\ell}\subset W^{\ell}$ can be viewed as the subspace of all functions
from$~W^{\ell}$ that are continuous across the interface $\Gamma^{\ell}$.
Define $U_{I}^{\ell}\subset U^{\ell}$ as the subspace of functions that are
zero on$~\Gamma^{\ell}$, i.e., the functions \textquotedblleft
interior\textquotedblright\ to the level$~\ell$ substructures. Define
projections $P^{\ell}$:$w^{\ell}\in W^{\ell}\longmapsto\left(  u_{I}^{\ell
},p_{I}^{\ell}\right)  \in\left(  U_{I}^{\ell},Q_{I}^{\ell}\right)  $ such
that
\begin{align*}
a\left(  u_{I}^{\ell},v_{I}^{\ell}\right)  +b\left(  v_{I}^{\ell},p_{I}^{\ell
}\right)   &  =a\left(  w^{\ell},v_{I}^{\ell}\right)  ,\quad\forall
v_{I}^{\ell}\in U_{I}^{\ell}\\
b\left(  u_{I}^{\ell},q_{I}^{\ell}\right)   &  =b\left(  w^{\ell},q_{I}^{\ell
}\right)  ,\quad\forall q_{I}^{\ell}\in Q_{I}^{\ell}.
\end{align*}
Functions from the nullspace of $P^{\ell}$ will be called Stokes harmonic on
level $\ell$. Next, let $\widehat{W}^{\ell}$ be the space of all Stokes
harmonic functions that are continuous across substructure interfaces\ on
level$~\ell$, and such that
\begin{equation}
U^{\ell}=U_{I}^{\ell}\oplus\widehat{W}^{\ell},\quad\text{and\quad}U_{I}^{\ell
}\perp_{a}\widehat{W}^{\ell}. \label{eq:int-harm-dec-levels}%
\end{equation}

Let $\widetilde{W}^{\ell}\subset W^{\ell}$ be the subspace of all functions
such that the values of any flux coarse degrees of freedom on level$~\ell$
have a common value over a face shared by a pair of adjacent level$~\ell
$\ substructures and vanish on $\partial\Omega_{i}^{\ell}\cap\partial\Omega$.
Define $\widetilde{W}_{\Pi}^{\ell}\subset\widetilde{W}^{\ell}$ as the subspace
of all functions such that their level $\ell$\ flux coarse degrees of freedom
between adjacent substructures coincide, and such that they are Stokes
harmonic, and let us also define $\widetilde{W}_{\Delta}^{\ell}\subset
W^{\ell}$ as the subspace of all functions such that their level $\ell$ flux
coarse degrees of freedom vanish. Clearly, functions in $\widetilde{W}_{\Pi
}^{\ell}$ are uniquely determined by the values of their level$~\ell$\ coarse
degrees of freedom, and
\begin{equation}
\widetilde{W}^{\ell}=\widetilde{W}_{\Delta}^{\ell}\oplus\widetilde{W}_{\Pi
}^{\ell}. \label{eq:tilde-dec-levels}%
\end{equation}
Let $E^{\ell}$ be a projection from $\widetilde{W}^{\ell}$ onto $U^{\ell}$,
defined by taking some weighted average of corresponding coarse degrees of
freedom on$~\Gamma^{\ell}$, cf. Remark~\ref{rem:averaging}$.$

These spaces and operators are used in both, Nested and Multilevel BDDC,
algorithms described below. Their hierarchy is shown concisely in
Figure~\ref{fig:multi-bddc}. We are now ready to generalize the two-level
Algorithm~\ref{alg:two-scale} to multiple levels.

\begin{algorithm}
[Nested BDDC]\label{alg:multiscale}
Find the solution $\left(  u^{1},p^{1}\right)  \in\left(  U^{1},Q^{1}\right)
$\ of the problem (\ref{eq:variational-level1-1}%
)-(\ref{eq:variational-level1-2}) in the following steps:

\noindent\textbf{for} $\ell=1,\ldots L-1$\textbf{,}

\begin{description}
\item Step 1: formulate the coarse problem as: find $\left(  w_{\Pi}^{\ell
},p_{0}^{\ell}\right)  \in\left(  \widetilde{W}_{\Pi}^{\ell},Q_{0}^{\ell
}\right)  $\ such that%
\begin{align}
a\left(  w_{\Pi}^{\ell},z_{\Pi}^{\ell}\right)  +b\left(  z_{\Pi}^{\ell}%
,p_{0}^{\ell}\right)   &  =0,\qquad\forall z_{\Pi}^{\ell}\in\widetilde{W}%
_{\Pi}^{\ell},\label{eq:alg-nested-coarse-1}\\
b\left(  w_{\Pi}^{\ell},q_{0}^{\ell}\right)   &  =\left\langle f^{\ell}%
,q_{0}^{\ell}\right\rangle ,\qquad\forall q_{0}^{\ell}\in Q_{0}^{\ell},
\label{eq:alg-nested-coarse-2}%
\end{align}

\item If $\ell=L-1$, solve the coarse problem directly, drop $p_{0}^{\ell}$,
and set $u_{0}^{L-1}=w_{\Pi}^{L-1}$.

\item Else, set $U^{\ell+1}=\widetilde{W}_{\Pi}^{\ell}$ and set up the
right-hand side of (\ref{eq:alg-nested-coarse-2}) for level $\ell+1$,%
\[
f^{\ell+1}\in Q^{\ell+1\prime},\quad\left\langle f^{\ell+1},q^{\ell
+1}\right\rangle =\left\langle f^{\ell},q^{\ell+1}\right\rangle ,\qquad\forall
q^{\ell+1}\in Q^{\ell+1},
\]

\end{description}

\noindent\textbf{end} \pagebreak

\noindent\textbf{for} $\ell=L-1,\ldots1$\textbf{,}

\begin{description}
\item Step 2: find the substructure components $\left(  u_{I}^{\ell}%
,p_{I}^{\ell}\right)  \in\left(  U_{I}^{\ell},Q_{I}^{\ell}\right)  $ from
\begin{align*}
a\left(  u_{I}^{\ell},v_{I}^{\ell}\right)  +b\left(  v_{I}^{\ell},p_{I}^{\ell
}\right)   &  =-a\left(  u_{0}^{\ell},v_{I}^{\ell}\right)  ,\qquad\forall
v_{I}^{\ell}\in U_{I}^{\ell},\\
b\left(  u_{I}^{\ell},q_{I}^{\ell}\right)   &  =\left\langle f^{\ell}%
,q_{I}^{\ell}\right\rangle -b\left(  u_{0}^{\ell},q_{I}^{\ell}\right)
,\qquad\forall q_{I}^{\ell}\in Q_{I}^{\ell},
\end{align*}
\newline drop $p_{I}^{\ell}$, and combine the two solutions
\[
u^{\ast,\ell}=u_{0}^{\ell}+u_{I}^{\ell}.
\]

\item Step 3: find the correction and the pressure $\left(  u_{\text{corr}%
}^{\ell},p^{\ell}\right)  \in\left(  U^{\ell},Q^{\ell}\right)  $ from%
\begin{align*}
a\left(  u_{\text{corr}}^{\ell},v^{\ell}\right)  +b\left(  v^{\ell},p^{\ell
}\right)   &  =-a\left(  u_{0}^{\ast,\ell},v^{\ell}\right)  ,\qquad\forall
v^{\ell}\in U^{\ell},\\
b\left(  u_{\text{corr}}^{\ell},q^{\ell}\right)   &  =0,\qquad\forall q^{\ell
}\in Q^{\ell}.
\end{align*}
Specifically, use the PCG method with the {M}ultilevel {BDDC} preconditioner
defined in Algorithm~\ref{alg:multilevel-bddc}, using the hierarchy of coarse
problems (\ref{eq:alg-nested-coarse-1})-(\ref{eq:alg-nested-coarse-2}).

\item Finally, combine the three solutions as
\[
u^{\ell}=u_{0}^{\ell}+u_{I}^{\ell}+u_{\text{corr}}^{\ell}.
\]

\item If $\ell>1$, drop $p^{\ell}$, and 
set $u_{0}^{\ell-1}=u^{\ell}$.

\end{description}

\noindent\textbf{end}
\end{algorithm}

We note that the first loop provides a natural approach of scaling-up through
the levels.
The Multilevel BDDC\ preconditioner used in Step~3 of
Algorithm~\ref{alg:multiscale} consists of recursive application of the
two-level BDDC\ preconditioner for the approximate solution of the hierarchy
of the coarse problems that were pre-computed in Step~1. Even though the
preconditioner differs only little from its original version for elliptic
problems described in~\cite[Algorithm~17]{Mandel-2008-MMB}, we again include
its saddle-point version here for completeness.

\begin{algorithm}
[Multilevel BDDC preconditioner]\label{alg:multilevel-bddc} Define the
preconditioner $\left(  r^{\ell},0\right)  \in\left(  U^{\ell\prime}%
,Q^{\ell\prime}\right)  \longmapsto\left(  u^{\ell},p^{\ell}\right)
\in\left(  U^{\ell},Q^{\ell}\right)  $\ as follows:

\noindent\textbf{for }$k=\ell,\ldots,L-1$\textbf{,}

\begin{description}
\item Compute the interior pre-correction $\left(  u_{I}^{k},p_{I}^{k}\right)
\in\left(  U_{I}^{k},Q_{I}^{k}\right)  $ from%
\begin{align}
a\left(  u_{I}^{k},v_{I}^{k}\right)  +b\left(  v_{I}^{k},p_{I}^{k}\right)   &
=\left\langle r^{k},v_{I}^{k}\right\rangle ,\qquad\forall v_{I}^{k}\in
U_{I}^{k},\label{eq:ML-I-precorr-1}\\
b\left(  u_{I}^{k},q_{I}^{k}\right)   &  =0,\qquad\forall q_{I}^{k}\in
Q_{I}^{k}. \label{eq:ML-I-precorr-2}%
\end{align}

\item Set up the updated residual%
\[
r_{B}^{k}\in U^{k\prime},\quad\left\langle r_{B}^{k},v^{k}\right\rangle
=\left\langle r^{k},v^{k}\right\rangle -\left[  a\left(  u_{I}^{k}%
,v^{k}\right)  +b\left(  v^{k},p_{I}^{k}\right)  \right]  ,\qquad\forall
v^{k}\in U^{k}.
\]

\item Compute the substructure correction $\left(  w_{\Delta}^{k},p_{I\Delta
}^{k}\right)  \in\left(  \widetilde{W}_{\Delta}^{k},Q_{I}^{k}\right)  $ from
\begin{align*}
a\left(  w_{\Delta}^{k},z_{\Delta}^{k}\right)  +b\left(  z_{\Delta}%
^{k},p_{I\Delta}^{k}\right)   &  =\left\langle r_{B}^{k},E^{k}z_{\Delta}%
^{k}\right\rangle ,\qquad\forall z_{\Delta}^{k}\in\widetilde{W}_{\Delta}%
^{k},\\
b\left(  w_{\Delta}^{k},q_{I}^{k}\right)   &  =0,\qquad\forall q_{I}^{k}\in
Q_{I}^{k}.
\end{align*}

\item Formulate the coarse problem as: find $\left(  w_{\Pi}^{k},p_{0}%
^{k}\right)  \in\left(  \widetilde{W}_{\Pi}^{k},Q_{0}^{k}\right)  $ such that
\begin{align}
a\left(  w_{\Pi}^{k},z_{\Pi}^{k}\right)  +b\left(  z_{\Pi}^{k},p_{0}%
^{k}\right)   &  =\left\langle r_{B}^{k},E^{k}z_{\Pi}^{k}\right\rangle
,\quad\forall z_{\Pi}^{k}\in\widetilde{W}_{\Pi}^{k},\label{eq:ML-coarse}\\
b\left(  w_{\Pi}^{k},q_{0}^{k}\right)   &  =0,\quad\forall q_{0}^{k}\in
Q_{0}^{k}. \label{eq:b_wPi_q0}%
\end{align}

\item If$\ k=L-1$, solve the coarse problem directly and set%
\begin{align*}
u^{L}  &  =w_{\Pi}^{L-1},\\
p^{L}  &  =p_{0}^{L-1}.
\end{align*}

\item Else, set $U^{k+1}=\widetilde{W}_{\Pi}^{k}$, set up the right-hand side
$r^{k+1}$ of (\ref{eq:ML-I-precorr-1}) for level$~k+1$,%
\[
r^{k+1}\in U^{k+1^{\prime}},\quad\left\langle r^{k+1},v^{k+1}\right\rangle
=\left\langle r_{B}^{k},E^{k}v^{k+1}\right\rangle ,\quad\forall v^{k+1}\in
U^{k+1},\label{eq:ML-ri+1}%
\]

\end{description}

\noindent\textbf{end}

\noindent\textbf{for }$k=L-1,\ldots,\ell\mathbf{,}$\textbf{\ }

\begin{description}
\item Average the approximate corrections,%
\begin{align}
u_{B}^{k}  &  =E^{k}\left(  w_{\Delta}^{k}+u^{k+1}\right)  ,\label{eq:uB^k}\\
p_{0}^{k}  &  =p^{k+1}.
\end{align}

\item Compute the interior post-correction $\left(  v_{I}^{k},q_{I}%
^{k}\right)  \in\left(  U_{I}^{k},Q_{I}^{k}\right)  $ from
\begin{align}
a\left(  v_{I}^{k},z_{I}^{k}\right)  +b\left(  z_{I}^{k},q_{I}^{k}\right)   &
=a\left(  u_{B}^{k},z_{I}^{k}\right)  ,\quad\forall z_{I}^{k}\in U_{I}%
^{k},\label{eq:ML-vIi}\\
b\left(  v_{I}^{k},\overline{q}_{I}^{k}\right)   &  =b\left(  u_{B}%
^{k},\overline{q}_{I}^{k}\right)  ,\quad\forall\overline{q}_{I}^{k}\in
Q_{I}^{k}. \label{eq:ML-pIi}%
\end{align}

\item Apply the combined corrections,
\begin{align}
u^{k}  &  =u_{I}^{k}+u_{B}^{k}-v_{I}^{k},\label{eq:u_k}\\
p^{k}  &  =p_{I}^{k}+p_{0}^{k}-q_{I}^{k}.
\end{align}

\end{description}

\noindent\textbf{end}
\end{algorithm}

In order to guarantee that the Multilevel BDDC\ preconditioner is invariant on
the space of divergence-free functions, we will need the following:

\begin{assumption}
\label{ass:enough-constraints}Suppose that the flux coarse degrees of freedom
are prescribed as averages over every face on every decomposition level$~\ell
$, $\ell=1,\dots,L-1$.
\end{assumption}

\begin{lemma}
\label{lem:constraints}
Let Assumption~\ref{ass:enough-constraints} be satisfied. Then,
\begin{align*}
b\left(  E^{\ell}w_{\Delta}^{\ell},q_{0}^{\ell}\right)   &  =0,\quad
\forall\left(  w_{\Delta}^{\ell},q_{0}^{\ell}\right)  \in\left(
\widetilde{W}_{\Delta}^{\ell},Q_{0}^{\ell}\right)  ,\\
b\left(  E^{\ell}w_{\Pi}^{\ell},q_{0}^{\ell}\right)   &  =b\left(  w_{\Pi
}^{\ell},q_{0}^{\ell}\right)  ,\quad\forall\left(  w_{\Pi}^{\ell},q_{0}^{\ell
}\right)  \in\left(  \widetilde{W}_{\Pi}^{\ell},Q_{0}^{\ell}\right)  .
\end{align*}
\end{lemma}

\begin{proof}
Note that with Assumption~\ref{ass:enough-constraints} satisfied, the
values of coarse degrees of freedom of functions from the space $\widetilde
{W}_{\Delta}^{\ell}$ are zero, i.e., the fine degrees of freedom have a zero average,
and the values of coarse degrees of\ freedom
for functions from the space $\widetilde{W}_{\Pi}^{\ell}$ for all (pairs of)
adjacent substructures coincide. The claim now follows from the divergence theorem,
because $q_0$ are piecewise constant in each level~$\ell$ subdomain separately,
cf. also~\cite[Lemma~2]{Tu-2005-BAM}.
\qed
\end{proof}

\begin{lemma}
\label{lem:divergence-free} Let Assumption~\ref{ass:enough-constraints} be
satisfied. Then 
the solution $u^\ell$ obtained
from the Multilevel BDDC\ preconditioner in Algorithm~\ref{alg:multilevel-bddc}
is divergence-free.
\end{lemma}

\begin{proof}
Let $\ell=1,\dots,L-1$ be fixed. Using (\ref{eq:uB^k}),
Lemma~\ref{lem:constraints} and (\ref{eq:b_wPi_q0}), we~get%
\begin{equation}
b\left(  u_{B}^{\ell},q_{0}^{\ell}\right)  =b\left(  E^{\ell}w^{\ell}%
,q_{0}^{\ell}\right)  =b\left(  w_{\Pi}^{\ell},q_{0}^{\ell}\right)
=0,\quad\forall q_{0}^{\ell}\in Q_{0}^{\ell},\label{eq:b_uB=b_wPi}%
\end{equation}
which also shows that $u_{B}^{\ell}\in\widehat{W}_{B}^{\ell}$. Next, using
(\ref{eq:u_k}) and (\ref{eq:Q-dec-levels}), we obtain
\[
b\left(  u^{\ell},q^{\ell}\right)  =b\left(  u_{I}^{\ell}+u_{B}^{\ell}%
-v_{I}^{\ell},q_{0}^{\ell}+q_{I}^{\ell}\right)  =0,\quad\forall q^{\ell}\in
Q^{\ell} ,%
\]
which follows using (\ref{eq:b-orth}), (\ref{eq:ML-I-precorr-2}),
(\ref{eq:ML-pIi}), and (\ref{eq:b_uB=b_wPi}), i.e., $u^{\ell}$ is
divergence-free. \qed
\end{proof}

Thus with a careful choice of the initial solution, such that the residual
corresponding to the substructure interiors and pressures is zero,
the output of the Multilevel BDDC preconditioner is divergence-free and by
induction all the PCG\ iterates, which are linear combinations of the initial
error and the outputs of the preconditioner, stay in the divergence-free subspace.

In order to provide the condition number bound of the Multilevel BDDC\ for a
saddle-point problem studied here, let us define, for levels $\ell
=1,\dots,L-1$, a~hierarchy of \emph{balanced} spaces
\[
\widetilde{W}_{B}^{\ell}=\left\{  w^{\ell}\in\widetilde{W}^{\ell}:b\left(
w^{\ell},q_{0}^{\ell}\right)  =0,\quad\forall q_{0}^{\ell}\in Q_{0}^{\ell
}\right\}  .
\]
The following condition number bound is a variant of~\cite[Lemma
20]{Mandel-2008-MMB}.

\begin{lemma}
\label{lem:bddc-ML-estimate}If for some $\omega^{\ell}\geq1$,
\begin{equation}
\left\Vert (I-P^{\ell})E^{\ell}w^{\ell}\right\Vert _{a}^{2}\leq\omega^{\ell
}\left\Vert w^{\ell}\right\Vert _{a}^{2},\quad\forall w^{\ell}\in\widetilde
{W}_{B,}^{\ell}\quad\ell=1,\ldots,L-1,\label{eq:est-omega-k}%
\end{equation}
then the Multilevel BDDC preconditioner (Algorithm~\ref{alg:multilevel-bddc}) satisfies $\kappa\leq%
{\textstyle\prod_{\ell=1}^{L-1}}
\omega^{\ell}.$
\end{lemma}

\begin{proof}
The bound was given for all $w^{\ell}\in$ $\widetilde{W}^{\ell}$ in the
context of scalar elliptic problems in~\cite[Lemma 20]{Mandel-2008-MMB}. Here, we
need to show that for any $w^{\ell}\in\widetilde{W}_{B}^{\ell}$, the bilinear
form $b$ will vanish also for the function on the left hand-side, i.e., that $\left(  I-P^{\ell
}\right)  E^{\ell}w^{\ell}\in\widetilde{W}_{B}^{\ell}$.
So, consider~(\ref{eq:tilde-dec-levels}) and let $w^{\ell
}=w_{\Delta}^{\ell}+w_{\Pi}^{\ell}$. Then
\[
b\left(  \left(  I-P^{\ell}\right)  E^{\ell}w^{\ell},q_{0}^{\ell}\right)
=b\left(  \left(  I-P^{\ell}\right)  w_{\Pi}^{\ell},q_{0}^{\ell}\right)
=b\left(  w_{\Pi}^{\ell},q_{0}^{\ell}\right)  =0,
\]
which follows from Lemma~\ref{lem:constraints},
definition of $P^{\ell}$ and (\ref{eq:b-orth}), and from (\ref{eq:b_wPi_q0}).
\qed
\end{proof}


\section{Condition number bound for the model problem}

\label{sec:condition}

We will now apply the methodology from~\cite{Mandel-2008-MMB} in order to
derive a condition number bound for the model problem with the lowest order
Raviart-Thomas dicretization. The key is the lower bound derived by
Tu~\cite{Tu-2011-TBA}, which is limited to a geometric decomposition of the
domain $\Omega$\ on every decomposition level. In particular, let us make the following:

\begin{assumption}
\label{ass:geometry}Each subdomain $\Omega_{i}^{\ell},$ $\ell=0,\dots
,L-1$\ and $i=1,\dots,N^{\ell}$\ is quadrilateral. The subdomains also form on
every decomposition level $\ell$\ a quasi-uniform coarse mesh of the domain
$\Omega$ with a characteristic mesh size $H^{\ell}$.
\end{assumption}

First, note that by (\ref{eq:b_wPi_q0}), on each level $\ell=0,\dots,L-1$,  
the coarse basis functions are \emph{balanced}, i.e., 
for all $w_\Pi \in \widetilde{W}_\Pi^\ell$ 
we have that 
\[
b \left( w_\Pi, q_0   \right) =0, \qquad \forall q_0\in Q_0^\ell, 
\]
and we can use the $a-$norm, which is also equivalent to $L^{2}-$norm, on the
space $\widetilde{W}_{\Pi}^{\ell}$. So, let $\left\Vert w\right\Vert
_{a(\Omega_{i}^{\ell})}$ be the energy norm of a function $w\in\widetilde{W}%
_{\Pi}^{\ell}$,\ $\ell=1,\ldots,L-1,$\ restricted to subdomain$~\Omega
_{i}^{\ell},$ $i=1,\ldots N^{\ell}$, and let $\left\Vert w\right\Vert _{a}%
$\ be the norm obtained by piecewise integration over each $\Omega_{i}^{\ell}%
$. To apply Lemma~\ref{lem:bddc-ML-estimate} to our model problem, we need to
generalize the polylogarithmic estimate from Theorem~\ref{thm:two-level-bound}
to coarse levels. To this end, let $I^{\ell+1}:\widetilde{W}_{\Pi}^{\ell
}\rightarrow\widetilde{U}^{\ell+1}$ be an interpolation from the level$~\ell$
coarse degrees of freedom (i.e., level $\ell+1$ degrees of freedom) to
functions in another space $\widetilde{U}^{\ell+1}$ and assume that, for all
levels $\ell=1,\ldots,L-1,$ and level $\ell$ subdomains $\Omega_{i}^{\ell}$,
$i=1,\ldots,N^{\ell},$ the interpolation satisfies for all $w\in
\widetilde{W}_{\Pi}^{\ell}$ and for all $\Omega_{i}^{\ell+1}$ the equivalence
\begin{equation}
c_{1}^{\ell}\left\Vert I^{\ell+1}w\right\Vert _{a(\Omega_{i}^{\ell+1})}%
^{2}\leq\left\Vert I^{\ell}w\right\Vert _{a(\Omega_{i}^{\ell+1})}^{2}\leq
c_{2}^{\ell}\left\Vert I^{\ell+1}w\right\Vert _{a(\Omega_{i}^{\ell+1})}^{2},
\label{eq:Ii-equiv}%
\end{equation}
with $c_{2}^{\ell}/c_{1}^{\ell}\leq\operatorname*{const}$ bounded
independently of $H^{0},\ldots,H^{\ell+1}$.

\begin{remark}
Since $I^{1}=I$, the two norms are the same on $\widetilde{W}_{\Pi}%
^{0}=\widetilde{U}^{1}=U^{1}.$
\end{remark}

For the three-level BDDC for saddle-point problems with the RT0 finite element
discretization in two dimensions, the result of Tu~\cite[Lemma~5.5]%
{Tu-2011-TBA}, can be written in our settings for all $w\in\widetilde{W}_{\Pi
}^{1}$ and for all $\Omega_{i}^{2}$\ as%
\begin{equation}
c_{1}^{1}\left\Vert I^{2}w\right\Vert _{a(\Omega_{i}^{2})}^{2}\leq\left\Vert
w\right\Vert _{a(\Omega_{i}^{2})}^{2}\leq c_{2}^{1}\left\Vert I^{2}%
w\right\Vert _{a(\Omega_{i}^{2})}^{2}, \label{eq:one-equiv-2D}%
\end{equation}
where $I^{2}$ is an interpolation from the coarse degrees of freedom given by
the averages over substructure faces, and $c_{2}^{1}/c_{1}^{1}\leq
\operatorname*{const}$ independently of $H/h$. We note that the level~$2$
substructures are called subregions in~\cite{Tu-2011-TBA} and $I^{1}=I$.

The assumption~(\ref{eq:Ii-equiv})
allows us to generalize the polylogarithmic estimate from
Theorem~\ref{thm:two-level-bound} to coarse levels using the same approach as
in~\cite[Section 7]{Mandel-2008-MMB}.

\begin{lemma}
\label{lem:W_i operator equiv}For all substructuring levels $\ell$
$=1,\ldots,L-1$,
\begin{equation}
\left\Vert (I-P^{\ell})E^{\ell}w^{\ell}\right\Vert _{a}^{2}\leq C_{\ell
}\left(  1+\log\frac{H^{\ell}}{H^{\ell-1}}\right)  ^{2}\left\Vert w^{\ell
}\right\Vert _{a}^{2},\quad\forall w^{\ell}\in\widetilde{W}_{B}^{\ell
}.\label{eq:ML-bound}%
\end{equation}
\end{lemma}

\begin{remark}
Variants of Lemma~\ref{lem:W_i operator equiv} can be found in two special cases corresponding to
$\ell=1$ and $\ell=2$
in~\cite{Tu-2011-TBA} as Lemma~5.6 and Lemma~5.8, respectively.
\end{remark}

Comparing Lemma~\ref{lem:W_i operator equiv}\ to
Lemma~\ref{lem:bddc-ML-estimate} with $\omega^{\ell}=C_{\ell}\left(
1+\log\frac{H^{\ell}}{H^{\ell-1}}\right)  ^{2}$ we get:

\begin{theorem}
\label{thm:ML-bound}
Let Assumptions~\ref{ass:enough-constraints} and~\ref{ass:geometry} be satisfied.
Then the Multilevel BDDC peconditioner from Algorithm~\ref{alg:multilevel-bddc}
for the model saddle-point problem in
2D with RT0 finite element discretization satisfies the condition number estimate%
\[
\kappa\leq%
{\textstyle\prod_{\ell=1}^{L-1}}
C_{\ell}\left(  1+\log\frac{H^{\ell}}{H^{\ell-1}}\right)  ^{2}.
\]
\end{theorem}


\begin{remark}
For $L=3$ we recover the estimate by Tu~\cite[Theorem~6.2]{Tu-2011-TBA}.
We also  note that the constants $C_\ell$ in the bound depend in general on the spatial variation of the coefficient~$k$,
cf. numerical experiments in Section~\ref{sec:numerical}.
\end{remark}

\begin{corollary}
In the case of uniform coarsening, i.e. with $H^{\ell}/H^{\ell-1}=H/h$ and the
same geometry of decomposition on all levels $\ell=1,\ldots L-1,$ we get
\begin{equation}
\kappa\leq C^{L-1}\left(  1+\log H/h\right)  ^{2\left(  L-1\right)
}.\label{eq:all-same}%
\end{equation}
\end{corollary}

\section{Numerical experiments}

\label{sec:numerical}

Numerical examples are presented for a Darcy's problem on a square domain in
2D discretized by the lowest order quadrilateral Raviart-Thomas finite
elements (RT0). A square domain was uniformly divided into substructures with
fixed $H^{\ell}/H^{\ell-1}$ ratio on each level $\ell$. The boundary
conditions did not allow any flux across the boundary.
The right-hand side was given by a unit source and sink in two distant corners
of the domain, so that the compatibility condition~(\ref{eq:compatibility})
was satisfied. The method has been implemented in Matlab and for the
preconditioned gradients we have used zero initial guess and stopping
criterion for a relative residual tolerance of $10^{-6}$. The results for
different coarsening ratios $H^{\ell}/H^{\ell-1}$ (the relative subdomain
size) and varying number of outer iterations given by the number of levels
$L$, are reported in Table~\ref{tab:1}. For each $L$, there were $L-1$ outer
iterations $\ell$, i.e., $\ell=1, \dots, L-1$, consisting of the three steps
described in Algorithms~\ref{alg:two-scale} and~\ref{alg:multiscale}. In the
third step the flux correction was computed by PCG with the $(\ell+1)$-level
BDDC preconditioner.

In the first set of experiments, the coefficient is set $k=1$. In this case,
the two choices of scaling in the averaging operator $E$, cf.
Remark~\ref{rem:averaging}, are exactly the same. From the results in
Table~\ref{tab:1} we can observe that with increasing number of levels, the
growth of the condition number is consistent with the prediction of
Theorem~\ref{thm:ML-bound} and in particular with formula (\ref{eq:all-same}).
Also, it appears that for a fixed number of levels the condition number grows
only mildly with increasing relative subdomain size given by the $H^{\ell
}/H^{\ell-1}$ ratio.

\begin{figure}[ptbh]
\begin{center}
\includegraphics[width=12cm]{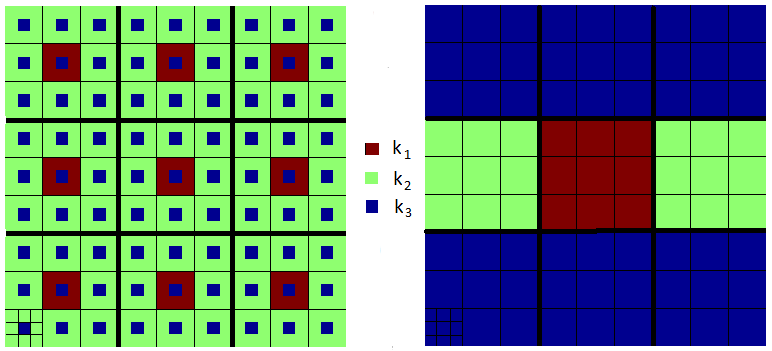}
\end{center}
\caption{The setup for the two experiments with variations in coefficients
$k_{1}$, $k_{2}$ and $k_{3}$. In both cases we have used the four-level method
with $H^{\ell}/H^{\ell-1}=3$. The pictures show three levels of decomposition
into subdomains with the first level decomposition shown only for one
level~$2$ subdomain, and the level of finite elements is not shown. The
picture on the left shows the case when the coefficient variations are
``interior'' to the substructures on the top level, and the jumps in
coefficients are aligned with the boundaries of substructures on lower levels.
The picture on the right shows the case when the jumps in coefficients are
aligned with the top level subdomain boundaries, and there are no ``interior''
variations. }%
\label{fig3}%
\end{figure}

In the second set of experiments, we have used the $\rho-$scaling and
experimented with jumps in the coefficient$~k$. In particular we have
performed two sets of experiments, both with the four level method and with
$H^{\ell}/H^{\ell-1}=3$, $\ell=0,\dots,3$, see Figure~\ref{fig3}. In the first
experiment, the coefficient variations were \textquotedblleft
interior\textquotedblright\ to the substructures on the top level, and the
jumps in coefficients were aligned with the substructure boundaries on lower
levels. In the second experiment, the jumps in coefficients were aligned with
the top level subdomain boundaries, and there were no \textquotedblleft
interior\textquotedblright\ coefficient variations. In both experiments we
have kept the coefficient $k_{2}$ fixed as $k_{2}=1$, and varied $k_{1}$ up to
$10^{2}$ and $k_{3}$ to as low as $10^{-2}$ in order to obtain a coefficient
jump of maximum order$~10^{4}$. The iteration counts in all cases were nearly
the same (with $2-3$ additional iterations) compared to those in
Table~\ref{tab:1}. The results thus indicate that the convergence is
independent of such jumps, which is also consistent (for the second setup)
with the observations of Tu~\cite{Tu-2011-TBA} for the three-level BDDC\ method.

It thus appears that the Nested BDDC\ method can be also used for problems
with variations of coefficients over multiple scales, if one is able to
perform a somewhat special partitioning into subdomains. However, because we
feel that this prevents a practical use of the proposed method for a realistic
simulations with coefficient variations that might not be exactly aligned with
the subdomain boundaries, we will address this issue in a separate study.

\begin{table}[ptbh]
\caption{The number of PCG iterations of the Multilevel BDDC preconditioner
from Algorithm~\ref{alg:multiscale} for different relative subdomain sizes
$H^{\ell}/H^{\ell-1}$, and different number of decomposition levels~$L$ which
determines the number of iterations of the Nested BDDC from
Algorithm~\ref{alg:multilevel-bddc}. For each decomposition level $\ell=1,
\dots, L-1$, nsub is the number of subdomains, $n$ is the total number of
degrees of freedom, $n_{\Gamma}$ is the number of degrees of freedom on the
interfaces, iter is the number of PCG iterations with the $M$-level BDDC
preconditioner where $M=L-\ell+1$. The stopping tolerance is $10^{-6}$, and
cond is the condition number estimate from the L\'{a}nczos sequence in
conjugate gradients.}%
\label{tab:1}
\begin{center}%
\begin{tabular}
[c]{|c|c|c|c|c|c|c|c|}\hline
$L$ & $\ell$ & $M$ & nsub & $n$ & $n_{\Gamma}$ & iter & cond\\\hline
\multicolumn{8}{|c|}{$H_{\ell}/H_{\ell-1}=3$}\\\hline
2 & 1 & 2 & 9 & 261 & 36 & 4 & 1.22\\\hline
\multirow{2}{*}{3} & 2 & 2 & 9 & 225 & 36 & 3 & 1.14\\
& 1 & 3 & 81 & 2241 & 432 & 8 & 2.07\\\hline
\multirow{3}{*}{4} & 3 & 2 & 9 & 225 & 36 & 3 & 1.14\\
& 2 & 3 & 81 & 2133 & 432 & 7 & 1.84\\
& 1 & 4 & 729 & 19,845 & 4212 & 11 & 3.48\\\hline
\multirow{4}{*}{5} & 4 & 2 & 9 & 225 & 36 & 3 & 1.14\\
& 3 & 3 & 81 & 2133 & 432 & 7 & 1.83\\
& 2 & 4 & 729 & 19,521 & 4212 & 10 & 3.09\\
& 1 & 5 & 6561 & 177,633 & 38,880 & 14 & 5.98\\\hline
\multicolumn{8}{|c|}{$H_{\ell}/H_{\ell-1}=4$}\\\hline
2 & 1 & 2 & 16 & 800 & 96 & 6 & 1.94\\\hline
\multirow{2}{*}{3} & 2 & 2 & 16 & 736 & 96 & 5 & 1.73\\
& 1 & 3 & 256 & 12,416 & 1920 & 10 & 3.45\\\hline
\multirow{3}{*}{4} & 3 & 2 & 16 & 736 & 96 & 5 & 1.72\\
& 2 & 3 & 256 & 12,160 & 1920 & 9 & 3.11\\
& 1 & 4 & 4096 & 197,120 & 32,256 & 14 & 6.62\\\hline
\multicolumn{8}{|c|}{$H_{\ell}/H_{\ell-1}=6$}\\\hline
2 & 1 & 2 & 36 & 3960 & 360 & 9 & 2.57\\\hline
\multirow{2}{*}{3} & 2 & 2 & 36 & 3816 & 360 & 9 & 2.30\\
& 1 & 3 & 1296 & 140,400 & 15,120 & 13 & 5.60\\\hline
\multicolumn{8}{|c|}{$H_{\ell}/H_{\ell-1}=8$}\\\hline
2 & 1 & 2 & 64 & 12,416 & 896 & 10 & 3.00\\\hline
\multirow{2}{*}{3} & 2 & 2 & 64 & 12,160 & 896 & 10 & 2.72\\
& 1 & 3 & 4096 & 787,456 & 64,512 & 17 & 7.46\\\hline
\multicolumn{8}{|c|}{$H_{\ell}/H_{\ell-1}=16$}\\\hline
2 & 1 & 2 & 256 & 197,120 & 7680 & 13 & 4.09\\\hline
\multicolumn{8}{|c|}{$H_{\ell}/H_{\ell-1}=32$}\\\hline
2 & 1 & 2 & 1024 & 3,147,776 & 63,488 & 15 & 5.25\\\hline
\end{tabular}
\end{center}
\end{table}

\begin{acknowledgements}
I would like to thank Dr. Christopher Harder and Prof. Jan Mandel for many discussions over the paper,
and the referees for useful comments and suggestions.
\end{acknowledgements}

\bibliographystyle{spmpsci}
\bibliography{bddc_mixed}

\end{document}